\documentclass{amsart}

\usepackage{ifpdf}

\usepackage{color}
\usepackage{enumitem}
\usepackage{amsmath,amssymb,latexsym}
\usepackage{graphicx}
\ifpdf
    \usepackage[colorlinks,citecolor=blue,linkcolor=blue]{hyperref}
\fi
\usepackage[initials]{amsrefs}
\usepackage[nocompress,nosort]{cite}

\newtheorem{theorem}{Theorem}[section]
\newtheorem{lemma}[theorem]{Lemma}
\newtheorem{prop}[theorem]{Proposition}
\newtheorem{cor}[theorem]{Corollary}

\theoremstyle{definition}
\newtheorem{definition}[theorem]{Definition}

\theoremstyle{remark}
\newtheorem{remark}[theorem]{Remark}

\numberwithin{equation}{section}

\newcommand{\interior}{\operatorname{int}}
\newcommand{\Rn}{\mathbb{R}^n}

\newcommand{\ang}[1]{\left\langle{#1}\right\rangle}
\newcommand{\supp}{\operatorname{supp}}

\begin{document}

\title[The one-phase Stefan problem]{Long-time behavior of the one-phase Stefan problem in periodic and random media}

\author{Norbert Po\v z\'ar}
\address[N. Po\v{z}\'ar]{Falculty of Mathematics and Physics, Institute of Science and Engineering, Kanazawa University, Kakuma, Kanazawa, 920-1192, Japan}
\curraddr{}
\email{npozar@se.kanazawa-u.ac.jp }
\thanks{}

\author{Giang Thi Thu Vu}
\address[G. T. T. Vu] {Graduate School of Natural Science and Technology, Kanazawa University, Kakuma, Kanazawa, 920-1192, Japan}
\curraddr{}
\email{vtgiang@vnua.edu.vn}
\thanks{}


\date{\today}
\begin{abstract}
  We study the long-time behavior of solutions of the one-phase Stefan problem in inhomogeneous
media in dimensions $n \geq 2$. Using the technique of rescaling which is consistent with the
evolution of the free boundary, we are able to show the homogenization of the free boundary
velocity as well as the locally uniform convergence of the rescaled solution to a self-similar
solution of the homogeneous Hele-Shaw problem with a point source. Moreover, by viscosity solution
methods, we also deduce that the rescaled free boundary uniformly approaches a sphere with respect
to Hausdorff distance.
\end{abstract}

\maketitle
\section{Introduction}

We consider the one-phase Stefan problem in periodic and random media in a dimension $n \geq 2$. The
aim of this paper is to understand the behavior of the solutions and their free boundaries when time $t \rightarrow \infty$.

Let $K \subset \mathbb{R}^n$ be a compact set with sufficiently regular boundary, for instance
$\partial K \in C^{1,1}$, and assume that $0 \in \interior K$. The one-phase
Stefan problem (on an exterior domain) with inhomogeneous latent heat of phase transition is to find a
function $v(x,t): \mathbb{R}^n \times [0,\infty) \rightarrow [0,\infty)$ that satisfies the free
boundary problem
\begin{equation}
  \label{Stefan}
  \left\{\begin{aligned}
    v_t- \Delta v &=0 && \text{ in } \{v>0\} \backslash K,\\
    v &=1 && \text{ on } K,\\
    V_\nu &=g(x)|Dv|&& \text{ on }\partial \{v>0\},\\
    v(x,0)&=v_0 &&\text{ on } \mathbb{R}^n,
    \end{aligned}\right.
\end{equation}
where $D$ and $\Delta$ are respectively the spatial gradient and Laplacian, $v_t$ is the partial
derivative of $v$ with respect to time variable $t$,  $V_\nu$ is the normal velocity of the
\textit{free boundary} $\partial \{v>0\}$. $v_0$ and $g$ are given functions, see below. Note that
the results in this paper can be trivially extended
to general time-independent positive continuous boundary data, $1$ is taken only to simplify the exposition.

The one-phase Stefan problem is a mathematical model of phase transitions between a solid and a liquid.
A typical example is the melting of a body of ice maintained at temperature $0$, in contact with a
region of water. The unknowns are the temperature distribution $v$ and its free boundary
$\partial \{v(\cdot, t)>0\}$, which models the ice-water interface. Given an initial temperature
distribution of the water, the diffusion of heat in a medium by conduction and the exchange of
latent heat will govern the system. In this paper, we consider an inhomogeneous medium
where the latent heat of phase transition, $L(x)= 1/g(x)$, and hence the velocity law depend on
position.  The related Hele-Shaw
problem is usually referred to in the literature as the quasi-stationary limit of the one-phase
Stefan problem when the heat operator is replaced by the Laplace operator. This problem typically
describes the flow of an injected viscous fluid between two parallel plates which form the
so-called Hele-Shaw cell, or the flow in porous media.

In this paper, we assume that the function $g$ satisfies the following two conditions, which guarantee respectively the
well-posedness of \eqref{Stefan} and averaging behavior as $t \to \infty$:
\begin{enumerate}
  \item \label {condition in g} $g$ is a Lipschitz function in $\mathbb{R}^n$, $m\leq g \leq M$ for some positive constants $m$ and $M$.
  \item \label {condition in g 2}$g(x)$ has some averaging properties so that Lemma~\ref{media}
    applies, for instance, one of the following holds:
  \begin{enumerate}
    \item $g$ is a $\mathbb{Z}^n$-periodic function,
    \item $g(x, \omega): \Rn \times A \to [m, M]$ is a stationary ergodic random variable over a probability space $(A, \mathcal{F},P)$.
  \end{enumerate}
\end{enumerate}
For a detailed definition and overview of stationary ergodic media, we refer to \cite{P1, K3} and
the references therein.

Throughout most of the paper we will assume that the initial data $v_0$ satisfies
\begin{equation}
\label{initial data}
\begin{aligned}
&v_0 \in C^2(\overline{\Omega_0 \backslash K}), v_0 > 0 \text{ in } \Omega_0, v_0=0, \mbox{ on } \Omega_0^c := \Rn \setminus
\Omega_0,
\mbox{ and } v_0=1 \mbox{ on } K,\\& |Dv_0| \neq 0 \text{ on } \partial \Omega_0,
\text{ for some bounded domain $\Omega_0 \supset K$.}
\end{aligned}
\end{equation}
This will guarantee the existence of both the weak and viscosity solutions below and their coincidence,
as well as the weak monotonicity \eqref{monotonicity condition}. However, the asymptotic
limit, Theorem~\ref{th:main-convergence}, is independent of the initial data, and therefore the result applies to arbitrary initial
data as long as the (weak) solution exists, satisfies the comparison principle, and the initial
data can be approximated from below and from above by data satisfying \eqref{initial data}. For
instance, $v_0 \in C(\Rn)$, $v_0 = 1$ on $K$, $v_0 \geq 0$, $\supp v_0$ compact is sufficient.

The Stefan problem \eqref{Stefan} does not necessarily have a global classical solution in $n
\geq 2$ as singularities of the free boundary might develop in finite time. The classical approach to
define a generalized solution is to integrate $v$ in time and introduce $u(x,t) := \int_{0}^{t}v(x,s)ds$
\cite{Baiocchi,Duvaut,FK,EJ,R1,R2,RReview}. If $v$ is sufficiently regular, then  $u$ solves the
variation inequality
\begin{equation}
\label{obstacle problem}
\begin{cases}
u(\cdot,t) \in \mathcal{K}(t),\\
(u_t - \Delta u)(\varphi - u) \geq f(\varphi -u) \mbox{ a.e } (x,t) \mbox{ for any } \varphi \in
\mathcal{K}(t),
\end{cases}
\end{equation}
where $\mathcal{K}(t)$ is a suitable functional space specified later in Section~\ref{sec:weak-sol} and $f$ is
\begin{equation}
  \label{f}
f(x)= \begin{cases}
  v_0(x), & v_0(x) > 0,\\
  -\displaystyle \frac{1}{g(x)}, & v_0(x) = 0.
\end{cases}
\end{equation}
This parabolic inequality always has a global unique solution $u(x,t)$ for initial data satisfying
(\ref{initial data}) \cite{FK,R1,R2,RReview}. The corresponding time derivative $v=u_t$, if it exists, is then called a \emph{weak solution} of the Stefan problem (\ref{Stefan}). The main advantage of this
definition is that the powerful theory of variational inequalities can be applied for the study of the
Stefan problem, and as was observed in \cite{R3,K2,K3} yields homogenization of \eqref{obstacle
problem}.

More recently, the notion of viscosity solutions of the Stefan problem was introduced and
well-posedness was established by Kim
\cite{K1}. Since this notion relies on the
comparison principle instead of the variational structure, it allows for more general, fully
nonlinear parabolic operators and boundary
velocity laws. Moreover, the pointwise viscosity methods seem more appropriate for studying the
behavior of the free boundaries. The natural question whether the weak and viscosity solutions
coincide was answered positively by Kim and Mellet \cite{K3} whenever the weak solution
exists.  In this paper we will use the strengths of both the weak and viscosity
solutions to study the behavior of the solution and its free boundary for large times.

 The homogeneous version of this problem, i.e, when $g\equiv const$, was studied by Quir\'os and
 V\'azques in \cite{QV}. They obtained the result on the long-time convergence of weak solution of
 the one-phase Stefan problem to the self-similar solution of the Hele-Shaw problem. The
 homogenization of this type of problem was considered by Rodrigues in \cite{R3} and by Kim-Mellet
 in \cite{K2,K3}. The long-time behavior of solution of the Hele-Shaw problem was studied in detail
 by the first author in \cite{P1}. In particular, the rescaled solution of the inhomogeneous
 Hele-Shaw problem converges to the self-similar solution of the Hele-Shaw problem with a point-source, formally
\begin{equation}
  \label{hs-point-source}
  \left\{\begin{aligned}
-\Delta v &=C\delta &&\mbox{ in } \{v>0\},\\
v_t&=\frac{1}{\left< 1/g \right>} |Dv|^2 &&\mbox{ on } \partial \{v>0\},\\
v(\cdot, 0)&=0,
\end{aligned}\right.
\end{equation}
 where $\delta$ is the Dirac $\delta$-function, $C$ is a constant depending on $K$ and $n$, and the constant $\left<1/g\right>$ will be properly defined later. Moreover, the rescaled free boundary uniformly approaches a sphere.

Here we extend the convergence result to the Stefan problem in the inhomogeneous medium. Since the
asymptotic behavior of radially symmetric solutions of the Hele-Shaw and the Stefan problem are
similar and the solutions are bounded, we can take the limit $t \rightarrow \infty$ and
obtain the convergence for rescaled solutions and their free boundaries. However, solutions of the
Hele-Shaw problem have a very useful monotonicity in time which is missing in the Stefan problem.
We instead take advantage of \eqref{monotonicity condition} for regular initial data satisfying \eqref{initial data}.
This makes some steps more difficult. Moreover, the heat operator is not invariant under the
rescaling, unlike the Laplace operator. The rescaled parabolic equation becomes elliptic when $\lambda
\rightarrow \infty$, which causes some issues when applying parabolic Harnack's inequality, for
instance.
Following \cite{QV,P1} we use the natural rescaling of solutions of the form
\begin{align*}
v^\lambda(x,t) &:=\lambda^{(n-2)/n}v(\lambda^{1/n}x, \lambda t) &&\mbox{ if } n \geq 3,\\
\intertext{and the corresponding rescaling for variational solutions}
u^\lambda(x,t) &:=\lambda^{-2/n}u(\lambda^{1/n},\lambda t) &&\mbox{ if } n\geq 3
\end{align*}
(see Section~\ref{sec:rescaling} for $n=2$). Then the rescaled viscosity solution satisfies the free boundary
velocity law
\begin{equation*}
V^\lambda_\nu=g(\lambda^{1/n}x)|Dv^\lambda|.
\end{equation*}

Heuristically, if $g$ has some averaging properties, such as in condition (\ref{condition in g 2}),
the free boundary velocity law should homogenize as $\lambda \to \infty$. Since the latent heat of
phase transition $1 /g$ should average out, the homogenized velocity law will be
\begin{equation*}
  V_\nu=\frac{1}{\left< 1/g\right>}|Dv|,
\end{equation*}
where $\left<1/g\right>$ represents the ``average" of $1/g$. More precisely, the quantity $\left <1/g\right >$ is the constant in the subadditive ergodic theorem such that
\begin{equation*}
  \int\limits_{\mathbb{R}^n}\frac{1}{g(\lambda^{1/n}x, \omega)}u(x)dx
  \rightarrow \int\limits_{\mathbb{R}^n}\left<\frac{1}{g}\right>u(x)dx  \text{ for all $u \in
L^2(\Rn)$}, \mbox{ for a.e. } \omega \in A.  \end{equation*}
In the periodic case, it is just the average of $1/g$ over one period.
Since we always work with $\omega \in A$ for which the convergence above holds, we omit it from the
notation in the rest of the paper.

This yields the first main result of this paper, Theorem \ref{convergence of variational
solutions}, on the homogenization of the obstacle problem (\ref{obstacle problem}) for the rescaled
solutions, with the correct singularity of the limit function at the origin, and therefore the
locally uniform convergence of variational solutions. To prove the second main result in Theorem
\ref{convergence of rescaled viscosity solution} on the locally uniform convergence of viscosity
solutions and their free boundaries, we use pointwise viscosity solution arguments. In summary,
we will show the following theorem.

\begin{theorem}
 \label{th:main-convergence}
   For almost every $\omega \in A$,
   the rescaled viscosity solution $v^\lambda$ of the Stefan problem \eqref{Stefan} converges
   locally uniformly to the unique self-similar solution $V$ of the Hele-Shaw problem
   \eqref{hs-point-source} in $(\mathbb{R}^n \backslash \{0\}) \times [0,\infty)$ as $\lambda
   \rightarrow \infty$, where $C$ depends only on $n$, the set $K$ and the boundary data $1$.
   Moreover, the rescaled free boundary $\partial \{(x,t):  v^\lambda(x,t) > 0\}$ converges to
   $\partial \{(x,t): V(x,t) > 0\}$ locally uniformly with respect to the Hausdorff distance.
\end{theorem}

It is a natural question to consider more general linear divergence form operators $\sum_{i,j}\partial_{x_i}
(a_{ij}(x) \partial_{x_j} \cdot )$ instead of the Laplacian in
\eqref{Stefan} so that the variational structure is preserved. This was indeed the setting
considered in \cite{K3}, with $g \equiv 1$ and appropriate free boundary velocity law adjusted for
the operator above. In the limit $\lambda \to \infty$, we expect that the rescaled solutions
$v^\lambda$ to converge to the unique solution of the Hele-Shaw type problem with a point source
with the homogenized non-isotropic operator with coefficients $\bar a_{i,j}$. This question is a topic of
ongoing work.

\subsection*{Context and open problems}
In recent years, there have been significant developments in the homogenization theory of partial
differential equations like Hamilton-Jacobi and second order fully nonlinear elliptic and parabolic
equations that have been made possible by the improvements of the viscosity solutions techniques, see
for instance the classical \cite{Evans,Souganidis,CSW,CS} to name a few.

A common theme of these results is finding (approximate) correctors and use the perturbed test
function method to establish the homogenization result in the periodic case, or using deeper
properties in the random case, such as the variational structure of the Hamilton-Jacobi equations
or the strong regularity results for elliptic and parabolic equations, including the ABP inequality.

One of the goals of this paper is to illustrate the powerful combination of variational and viscosity
solution techniques for some free boundary problems that have a variational structure. By viscosity
solution techniques we mean specifically pointwise arguments using the comparison principle.

Unfortunately, when the variational structure is lost, for instance, when the free boundary
velocity law is more general as in the problem with contact angle dynamics $V_\nu = |Dv| - g(x)$ so
that the motion is non-monotone
\cite{KimContact,KimContactRates}, or even
simple time-dependence $V_\nu = g(x,t) |Dv|$ \cite{Pozar15}, the comparison principle is all that is
left. Even in the periodic case, the classical correctors as solutions of a cell problem are not
available. This is in part the consequence of the presence of the free boundary on which the
operator is strongly discontinuous. \cite{KimHomog,KimContact,Pozar15} use a variant of the idea
that appeared in
\cite{CSW} to replace the correctors by solutions of certain obstacle problems. However, the
analysis of these solutions requires rather technical pointwise arguments since there are almost
no equivalents of the regularity estimates for elliptic equations.
An important tool in \cite{Pozar15} to overcome this was the large scale Lipschitz regularity of
the free boundaries of the obstacle problem solutions (called cone flatness there) that allows for
the control of the oscillations of the free boundary in the homogenization limit.

For the reasons above, the homogenization of free boundary problems is rather challenging and there are still many
open problems. Probably the most important one is the homogenization of free boundary
problems of the Stefan and Hele-Shaw type that do not admit a variational structure, such as those
mentioned above, in random
environments. Currently there is no known appropriate stationary subadditive quantity to which we could apply
the subadditive ergodic theorem to recover the homogenized free boundary velocity law, for
instance. Other tools like concentration inequalities have so far not yielded an alternative.

Another important problem are the optimal convergence rates of the free boundaries in the Hausdorff
distance. The techniques used in this paper do
not provide this information, however viscosity techniques were used to obtain non-optimal algebraic
convergence rates in \cite{KimContactRates}. It is an interesting question what the optimal rate in
the periodic case is, even for problems like \eqref{Stefan}. The large scale Lipschitz estimate
from \cite{Pozar15} could possibly directly give only $\varepsilon |\log \varepsilon|^{1/2}$-rate for
velocity law with $g(x/\varepsilon)$, but there are some indications that a rate $\varepsilon$ might be possible.

\subsection*{Outline}
The paper is organized as follows: In Section~\ref{sec:preliminaries}, we recall the definitions
and well-known results for weak and viscosity solutions. We also introduce the rescaling and state
some results for radially symmetric solutions. In Section~\ref{sec:var-conv}, we recall the limit
obstacle problem and prove the locally uniform convergence of rescaled variational solutions. In
Section~\ref{sec:visc-conv}, we focus on treating the locally uniform convergence of viscosity solutions and their free boundaries.

\section{Preliminaries}
\label{sec:preliminaries}
\subsection{Notation}

For a set $A$, $A^c$  is its complement. Given a nonnegative function $v$, we will use notations
for its positive set and free boundary of $v$,
$$\Omega(v):=\{(x,t): v(x,t)>0\}, \hspace{2cm} \Gamma(v):=\partial\Omega(v),$$
and for fixed time $t$,
$\hspace{1cm} \Omega_t(v):=\{x: v(x,t)>0\}, \hspace{1cm} \Gamma_t(v):=\partial\Omega_t(v) .$

$(f)_+$ is the positive part of $f$:
$(f)_+= \max(f,0)$.
\subsection{Weak solutions}
\label{sec:weak-sol}
Let $v(x,t)$ be a classical solution of the Stefan problem  (\ref{Stefan}). Fix $R, T > 0$ and set $B=B_R$
, $D=B \backslash K$. Following \cite{FK} it can be shown that, if $R$ is
large enough (depending on $T$), then the function $u(x,t):=\int_{0}^{t}v(x,s)ds$ solves the following variational problem: Find $u \in L^2(0,T; H^2(D))$ such that $u_t \in L^2(0,T;L^2(D))$ and
\label{Variational problem}
\begin{equation}
  \left\{\begin{aligned}
      u(\cdot,t) &\in \mathcal{K}(t), && 0 < t < T,\\
    (u_t-\Delta u)(\varphi -u) &\geq f(\varphi -u), &&\text{ a.e } (x,t) \in B \times (0,T) \text{ for any } \varphi \in \mathcal{K}(t),\\
    u(x,0)&=0 \text{ in } D.\\
\end{aligned}\right.
\end{equation}

Here we set
$\mathcal{K}(t)=\{\varphi \in H^1(D),\varphi \geq 0, \varphi =0 \text{ on } \partial B, \varphi =t \text{ on } K \}$
and $f$ was defined in \eqref{f}. We use the standard notation for Sobolev spaces $H^k$, $W^{k,p}$.
If $v$ is a classical solution of (\ref{Stefan}) then $u$ is solution of (\ref{Variational problem}),
but the inverse statement is not valid in general. However, we have the following result
\cite{FK,R1}.
\begin{theorem}[Existence and uniqueness of variational problem]
  If $v_0$ satisfies (\ref{initial data}), then the problem (\ref{Variational problem}) has a unique solution satisfying
  $$
  \begin{aligned}
    u &\in L^\infty(0,T; W^{2,p}(D)), \qquad 1\leq p \leq \infty,\\
  u_t &\in L^\infty(D \times (0,T)),\\
  \end{aligned}
  $$
  and
  $$
  \left\{\begin{aligned}
      u_t-\Delta u &\geq f && \mbox{ for a.e. } (x,t) \in \{u \geq 0 \},\\
  u(u_t- \Delta u-f)&=0 && \mbox{ a.e in } D\times (0, \infty).
  \end{aligned}\right.
  $$
\end{theorem}
We will thus say that if $u$ is a solution of (\ref{Variational problem}), then $u_t$ is a
\emph{weak solution} of the corresponding Stefan problem (\ref{Stefan}). The theory of variational
inequalities for an obstacle problem is well developed, for more details, we refer to \cite{FK,R1,K2}.
We now collect some useful results on the weak solutions from \cite{FK,R1}.
\begin{prop}
  The unique solution $u$ of (\ref{Variational problem}) satisfies
  $$0 \leq u_t \leq C \mbox{ a.e } D \times (0,T),$$
  where $C$ is a constant depending on $f$. In particular, $u$ is Lipschitz with respect to $t$ and
  $u$ is $C^{\alpha}(D)$ with respect to $x$ for all $\alpha \in (0,1)$. Furthermore, if $0 \leq t
  <s \leq T$, then $u(\cdot,t) < u(\cdot,s)$ in $\Omega_s(u)$ and also
  $\Omega_0 \subset \Omega_t(u) \subset \Omega_s(u)$.
\end{prop}
\begin{lemma}[Comparison principle for weak solutions]
  Suppose that $f \leq \hat{f}$. Let $u, \hat{u}$ be solutions of (\ref{Variational problem}) for respective $f, \hat{f}$. Then $u \leq \hat{u},$
  moreover,
  $$\theta \equiv \frac{\partial u}{\partial t} \leq \frac{ \partial \hat{u}}{\partial t} \equiv \hat{\theta}.$$
\end{lemma}
\begin{remark}
  Regularity of $\theta$ and its free boundary has been studied quite extensively, including Caffarelli and Friedman (see \cite{C,CF,FN}).
  It is known that a weak solution is classical as long as $\Gamma_t(u)$ has no singularity. The smoothness criterion (see \cite{C,FN}, \cite[Proposition 2.4]{QV}) immediately leads to the following corollary.
\end{remark}
\begin{cor}
  Radial weak solutions of the Stefan problem (\ref{Stefan}) are smooth classical solutions.
\end{cor}
\subsection{Viscosity solutions}
The second notion of solutions we will use are the viscosity solutions introduced in \cite{K1}.
First, for any nonnegative function $w(x,t)$ we define the semicontinuous envelopes
\begin{align*}
&w_\star(x,t) := \liminf_{(y,s) \rightarrow (x,t)}w(y,s), &w^\star(x,t) := \limsup_{(y,s) \rightarrow (x,t)}w(y,s).
\end{align*}
We will consider solutions in the space-time cylinder $Q:=(\mathbb{R}^n \backslash K) \times [0,\infty)$.
\begin{definition}
  \label{def of viscos subsol}
  A nonnegative upper semicontinuous function $v(x,t)$ defined in $Q$ is a viscosity subsolution of (\ref{Stefan}) if the following hold:
  \begin{enumerate}[label=\alph*)]
    \item For all $T \in (0,\infty)$, the set $\overline{\Omega (v)} \cap \{t \leq T\} \cap Q$ is bounded.
    \item For every $\phi \in C^{2,1}_{x,t}(Q)$ such that $v-\phi$ has a local maximum in $\overline{\Omega(v)} \cap \{t\leq t_0\} \cap Q$ at $(x_0,t_0)$, the following holds:
    \begin{enumerate}[label=\roman*)]
      \item If $v(x_0,t_0)>0$, then $(\phi_t- \Delta \phi )(x_0,t_0) \leq 0$.
      \item If $(x_0,t_0) \in \Gamma(v), |D\phi (x_0,t_0)| \neq 0$ and $(\phi_t-\Delta \phi )(x_0,t_0)>0$, then
      \begin{equation}
      (\phi_t-g(x_0)|D\phi|^2)(x_0,t_0) \leq 0.
      \end{equation}
    \end{enumerate}
  \end{enumerate}
  Analogously, a nonnegative lower semicontinuous function $v(x,t)$ defined in $Q$ is a viscosity
  supersolution if (b) holds with maximum replaced by minimum, and with inequalities reversed in
  the tests for $\phi$ in (i--ii). We do not need to require (a).
\end{definition}
Now let $v_0$ be a given initial condition with positive set $\Omega_0$ and free boundary $\Gamma_0=\partial \Omega_0$, we can define viscosity subsolution and supersolution of (\ref{Stefan}) with corresponding initial data and boundary data.
\begin{definition}
  A viscosity subsolution of (\ref{Stefan}) in $Q$ is a viscosity subsolution of (\ref{Stefan}) in $Q$ with initial data $v_0$ and boundary data $1$ if:
  \begin{enumerate}[label=\alph*)]
    \item $v$ is upper semicontinuous in $\bar Q, v=v_0$ at $t=0$ and $v \leq 1$ on $\Gamma$,
    \item $\overline{\Omega(v)}\cap \{t=0\}=\overline{\{x: v_0(x)>0\}} \times \{0\}$.
  \end{enumerate}
  A viscosity supersolution is defined analogously by requiring (a) with $v$ lower semicontinuous
  and $v \geq 1$ on $\Gamma$. We do not need to require (b).
\end{definition}
And finally we can define viscosity solutions.
\begin{definition}
  The function $v(x,t)$ is a viscosity solution of (\ref{Stefan}) in $Q$ (with initial data $v_0$ and boundary data $1$) if $v$ is a viscosity supersolution and $v^\star$ is a viscosity subsolution of (\ref{Stefan}) in $Q$ (with initial data $v_0$ and boundary data $1$).
\end{definition}
\begin{remark}
  By standard argument, if $v$ is the classical solution of (\ref{Stefan}) then it is a viscosity solution of that problem in $Q$ with initial data $v_0$ and boundary data $1$.
\end{remark}
The existence and uniqueness of a viscosity solution as well as its properties have been studied in
great detail in \cite{K1}. One important feature of viscosity solutions is that they satisfy a
comparison principle for ``strictly separated'' initial data.

One of the main tools we will use in this paper is the following coincidence of weak
and viscosity solutions from \cite{K3}.
 \begin{theorem}[cf. {\cite[Theorem 3.1]{K3}}]
   Assume that $v_0$ satisfies (\ref{initial data}). Let $u(x,t)$ be the unique solution of (\ref{Variational problem}) in $B \times [0,T]$ and let $v(x,t)$ be the solution of
   \begin{equation}
   \label{coincidence eq}
   \left\{
     \begin{aligned}
   v_t-\Delta v &=0 && \mbox{in } \Omega(u) \backslash K,\\
   v &=0 && \mbox{on } \Gamma(u),\\
   v &=1 &&\mbox{in } K,\\
   v(x,0) &=v_0(x).
 \end{aligned}
   \right.
   \end{equation}
   Then $v(x,t)$ is a viscosity solution of (\ref{Stefan}) in $B \times [0,T]$ with initial data
      $v(x,0)=v_0(x)$, and
    $u(x,t)= \int_{0}^{t}v(x,s)ds$.
 \end{theorem}
 \begin{remark}
   The definition of the solution $v$ of \eqref{coincidence eq} must be clarified
   when $\Omega(u)$ is not smooth. Since $u$ is continuous and $\Omega(u)$ is bounded at all times ( \cite[Lemma 3.6]{K3}) then the existence of solution of (\ref{coincidence eq}) is provided by Perron's method as
   \begin{equation*}
   v=\sup \{w| w_t- \Delta w \leq 0 \mbox{ in } \Omega(u), w\leq 0 \mbox{ on } \Gamma(u), w\leq 1 \mbox{ in } K, w(x,0) \leq v_0(x)\}.
   \end{equation*}
  Note that $v$ might be discontinuous on $\Gamma(u)$.
   \end{remark}
 The coincidence of weak and viscosity solutions gives us a more general comparison principle.
 \begin{lemma}[cf. {\cite[Corollary 3.12]{K3}}]
   Let $v^1$ and $v^2$  be, respectively, a viscosity subsolution and supersolution of the Stefan problem (\ref{Stefan}) with continuous initial data $v_0^1 \leq v_0^2$ and boundary data $1$. In addition, suppose that $v_0^1$(or $v_0^2$) satisfies condition (\ref{initial data}). Then
   $v^1_\star \leq v^2  \mbox{ and }  v^1 \leq (v^2)^\star \mbox{ in }  \mathbb{R}^n \backslash K
   \times [0,\infty).$
 \end{lemma}

 \subsection{Rescaling}
 \label{sec:rescaling}
 We will use the following rescaling of solutions as in \cite{P1}.
 \subsubsection{For $n \geq 3$}
 For $\lambda >0$ we use the rescaling
 \begin{align*}
 &v^\lambda(x,t)=\lambda^{\frac{n-2}{n}}v(\lambda^{\frac{1}{n}}x,\lambda t), & u^\lambda(x,t)= \lambda^{-\frac{2}{n}} u(\lambda^{\frac{1}{n}}x, \lambda t).
 \end{align*}
 If we define $K^\lambda := K / \lambda^{\frac{1}{n}}$ and $\Omega_0^\lambda:=
 \Omega_0/\lambda^{\frac{1}{n}}$ then $v^\lambda$ satisfies the problem
 \begin{equation}
 \label{rescaled equation}
 \left\{
   \begin{aligned}
 \lambda^{\frac{2-n}{n}}v_t^\lambda- \Delta v^\lambda &=0 && \mbox{ in } \Omega(v^\lambda) \backslash K^\lambda,\\
 v^\lambda &= \lambda^{\frac{n-2}{n}} && \mbox{ on } K^\lambda,\\
 v_t^\lambda&=g^\lambda(x) |Dv^\lambda|^2 && \mbox{ on } \Gamma(v^\lambda),\\
 v^\lambda(\cdot,0)&=v_0^\lambda,
 \end{aligned}
\right.
 \end{equation}
 where $g^\lambda(x)=g(\lambda^{\frac{1}{n}}x)$.
 And the rescaled $u^\lambda$ satisfies the obstacle problem
 \begin{equation}
 \label{rescaled Variational problem}
 \left\{
   \begin{aligned}
 u^\lambda(\cdot,t) &\in \mathcal{K}^\lambda(t),\\
 (\lambda^{\frac{2-n}{n}}u^\lambda_t-\Delta u^\lambda)(\varphi -u^\lambda) &\geq
 f(\lambda^{\frac{1}{n}}x)(\varphi -u^\lambda) &&\text{ a.e } (x,t) \in B_R \times (0,T)\\
 &&&\text{ for any } \varphi \in \mathcal{K}^\lambda(t),\\
 u^\lambda(x,0)&=0,
 \end{aligned}
 \right.
 \end{equation}
 where
 $\mathcal{K^\lambda}(t)=\{\varphi \in H^1(\mathbb{R}^n),\varphi \geq 0, \varphi =0 \text{ on } \partial B^\lambda, \varphi =\lambda^{\frac{n-2}{n}}t \text{ on } K^\lambda \}.$

 \subsubsection{For n=2}
 For dimension $n=2$, we use a different rescaling that preserves the singularity of logarithm,
 namely
 \begin{align}
 \label{rescaling n=2}
 &v^\lambda(x,t)=\log \mathcal{R}(\lambda)v(\mathcal{R}(\lambda)x, \lambda t),&
 u^\lambda(x,t)=\displaystyle \frac{\log \mathcal{R}(\lambda)}{\lambda}u(\mathcal{R}(\lambda)x,
 \lambda t),
 \end{align}
 where $\mathcal{R}(\lambda)$ is the unique solution of
 $\mathcal{R}^2 \log \mathcal{R}= \lambda$, $\lim_{\lambda \to \infty} \mathcal{R}(\lambda) \to
 \infty$ (see \cite{P1} for more details). $v^\lambda$ and $u^\lambda$ satisfy rescaled problems
 analogous to \eqref{rescaled equation} and \eqref{rescaled Variational problem}. In particular,
 the term $\lambda^{(2-n)/n}$ in front of the time derivatives is replaced by
 $1/\log(\mathcal{R}(\lambda)) \to 0$ as $\lambda \to \infty$.

 \subsection{Convergence of radially symmetric solutions} We will recall the results on the
 convergence of radially symmetric solutions of (\ref{Stefan}) as derived in \cite{QV}. First, we
 collect some useful facts of radial solution of the Hele-Shaw problem and then use a comparison to
 have the information of radial solution of the Stefan problem. The radially symmetric solution of
 the Hele-Shaw problem  in the domain $|x| \geq a, t \geq 0$ is a pair of functions $p(x,t)$ and
 $R(t)$, where $p$ is of the form
 \begin{equation}
 \label{radial Hele-Shaw}
 p(x,t)= \begin{cases}
 \displaystyle \frac{Aa^{n-2}\left(|x|^{n-2}-R^{n-2}(t)\right){+}}{a^{2-n}-R^{2-n}(t)}, &n \geq 3,\\
 \displaystyle \frac{A\left( \log\frac{R(t)}{|x|}\right)_{+}}{\log \frac{R(t)}{a}}, & n=2,
 \end{cases}
 \end{equation}
 and $R(t)$ satisfies a certain algebraic equation (see \cite{QV} for details).

 This solution satisfies the boundary conditions and initial conditions
 \begin{equation}
 \label{boundary con1}
 \begin{aligned}
p(x,t)&=Aa^{2-n} &&\mbox{ for } |x|=a>0,\\
p(x,t)&=0 &&\mbox{ for } |x| =R(t),\\
R'(t)&=\frac{1}{L} |Dp| &&\mbox{ for } |x| =R(t),\\
R(0)&=b >a.
 \end{aligned}
 \end{equation}
 Furthermore,
 \begin{align*}
 \lim\limits_{t \rightarrow \infty}\frac{R(t)}{c_\infty t^{1/n}}&=1,  & c_\infty &=
 \left(\frac{An(n-2)}{L}\right)^{1/n} &\mbox{ if } n \geq3,\\
 \lim\limits_{t \rightarrow \infty} \frac{R(t)}{c_\infty\left(t/\log t\right)^{1/2}} &=1, &
 c_\infty &= 2\sqrt{A/L} &\mbox{ if } n=2.
 \end{align*}
 In dimension $n=2$, we will also use $ \lim\limits_{t \rightarrow \infty} \frac{\log R(t)}{\log t}=\frac{1}{2}$.

 The radial solution of the Stefan problem satisfies the corresponding conditions similar to
 (\ref{boundary con1}) with the initial data
 \begin{equation}
 \label{initial cond radial}
 \theta(x,0) =\theta_0(|x|) \mbox{ if } |x| \geq a.
 \end{equation}
 The following results were shown in \cite{QV}.
 \begin{lemma}[cf. {\cite[Proposition~6.1]{QV}}]
Let $p$ and $\theta$ be radially symmetric solutions to the Hele-Shaw problem and to the Stefan
 problem respectively, and let $\{|x|=R_p(t)\}, \{|x|=R_\theta(t)\}$ be the corresponding interfaces. If $R_p(0)> R_\theta(0), p(x,0)\geq \theta(x,0)$ and, moreover, $p(x,t) \geq \theta(x,t)$ on the fixed boundary, that is, for $|x| =a, t>0$, then $p(x,t) \geq \theta(x,t)$ for all $|x| \geq a$ and $t \geq 0$.
\end{lemma}
This immediately leads to an upper bound for the free boundary of radial solutions of Stefan problem, see Corollary 6.2, Theorem 6.4, Theorem 7.1 in \cite{QV}.
\begin{lemma}
  \label{free boundary bound}
  Let $\{|x|=R(t)\}$ be the free boundary of a radial solution to the Stefan problem satisfying the corresponding conditions (\ref{boundary con1}) and (\ref{initial cond radial}). There are constants $C,T>0$, such that, for all $t \geq T$,
  \begin{equation*}
  \begin{matrix}
    &  R(t) \leq Ct^{1/n},  n \geq 3,& \mbox{ or }
    &  R(t) \leq C(t/ {\log t})^{1/2},  n=2.
  \end{matrix}
  \end{equation*}
  Moreover, we have
  \begin{equation*}
  \begin{matrix}
  &\lim\limits_{t \rightarrow \infty} \displaystyle \frac{R(t)}{t^{1/n}}=
\left(An(n-2)/L\right)^{1/n}, n \geq 3,& \mbox{ or }
  &\lim\limits_{t \rightarrow \infty} \displaystyle \frac{R(t)}{(t/ \log t)^{1/2}}= 2 \sqrt{A/L}, n=2.
  \end{matrix}
  \end{equation*}
\end{lemma}
The solution of the Stefan problem is bounded for all time.
\begin{lemma}[cf.{\cite[Lemma 6.3]{QV}}]
  \label{Boundedness of weak solution}
  Let $\theta$ be a weak solution of the Stefan problem for $n \geq 2$. There is a constant $C>0$
  such that, for all $t > 0$,
  $0 \leq \theta(x,t) \leq C |x|^{2-n}.$
\end{lemma}
 Next, we define the solution of the Hele-Shaw problem with a point source, which will appear as
 the limit function in our convergence results,
 \begin{equation}
 \label{Hele-Shaw with point sorce}
 V(x,t)=V_{A,L}(x,t)=
 \begin{cases}
 A\left(|x|^{2-n}- \rho^{2-n}(t)\right)_{+}, & n \geq 3,\\
 A\left(\log \frac{\rho(t)}{|x|}\right)_{+}, & n=2,
 \end{cases}
 \end{equation}
 where
 \begin{equation*}
 \rho(t)= \rho_L(t)=R_\infty=
 \begin{cases}
 \left(An(n-2)t/L\right)^{1/n}, & n \geq 3,\\
 \left(2At/L\right)^{1/2}, & n = 2.
 \end{cases}
 \end{equation*}
 It is the unique solution of the Hele-Shaw problem with a point source,
 \begin{equation}
 \label{Hele-Shaw point source problem}
\left\{
  \begin{aligned}
 \Delta v &= 0 && \mbox{in } \Omega(v) \backslash \{0\},\\
 \lim\limits_{|x| \rightarrow 0} \displaystyle \frac{v(x,t)}{|x|^{2-n}}&= C_*, && n\geq 3, \quad
 \mbox{or} &\lim\limits_{|x| \rightarrow 0} \displaystyle -\frac{v(x,t)}{\log(|x|)}&= C_*, && n=2,\\
 \displaystyle v_t&=\frac{1}{L}|Dv|^2 && \mbox{on } \partial \Omega(v),\\
 v(x,0)&=0 && \mbox{in } \mathbb{R}^n \backslash \{0\}.
 \end{aligned}
 \right.
 \end{equation}
 The asymptotic result for radial solutions of the Stefan problem follows from Theorem 6.5 and Theorem 7.2 in \cite{QV}.
 \begin{theorem}[Far field limit]
   \label{far field radial}
   Let $\theta$ be the radial solution of the Stefan problem satisfying boundary conditions (\ref{boundary con1}) and initial condition (\ref{initial cond radial}). Then
   \begin{equation}
   \lim\limits_{t \rightarrow \infty} t^{(n-2)/n} |\theta(x,t) - V(x,t)|=0
   \end{equation}
   uniformly on sets of form $\{x \in \mathbb{R}^n: |x| \geq \delta t^{1/n}\}, \delta >0$ if $n \geq 3$, and
   \begin{equation}
   \displaystyle \lim\limits_{t \rightarrow \infty} \log \sqrt{\frac{2A}{L}}\mathcal{R}(t) \left |\theta(x,t) - \displaystyle \frac{A}{\log \sqrt{\frac{2A}{L}}\mathcal{R}(t)}\left(\log  \sqrt{\frac{2A}{L}}\mathcal{R}(t)- \log|x|\right)_+\right |=0
   \end{equation}
   uniformly on sets of form $\{x \in \mathbb{R}^n: |x| \geq \delta \mathcal{R}(t)\}, \delta >0$ if $n =2 $.
 \end{theorem}
 \begin{proof}
   Follow the proof of Theorem 6.5 in \cite{QV} with recalling that we assume $\theta = Aa^{2-n}$ for $|x|=a$ we immediately get the result for $n=3$.

   For $n=2$, let $\mathcal{R}_1(t)$ be the solution of
$\frac{\mathcal{R}_1^2}{2}\left(\log \mathcal{R}_1- \frac{1}{2}\right)=\frac{At}{L}$ with
   $\lim\limits_{t \rightarrow \infty} \frac{\mathcal{R}_1(t)}{\mathcal{R}(t)}=\sqrt{\frac{2A}{L}}.$
   Thus, we can replace $\mathcal{R}_1(t)$ in Theorem 7.2 in \cite{QV} by $ \sqrt{\frac{2A}{L}}\mathcal{R}(t)$.
 \end{proof}
 Finally, we can improve Theorem \ref{far field radial} to have the following convergence result
 for rescaled radial solutions of the Stefan problem which holds up to $t=0$.
 \begin{lemma}[Convergence for radial case]
   \label{convergence radial lemma}
   Let $\theta(x,t)$ be a radial solution of the Stefan problem satisfying the corresponding boundary  and initial condition. Then $\theta^\lambda$ converge locally uniformly to $V_{A,L}$ in the set $(\mathbb{R}^n \backslash \{0\}) \times [0,\infty)$.
 \end{lemma}
 \begin{proof}
   We will prove the uniform convergence in the sets $Q=\{(x,t): |x| \geq \varepsilon, 0 \leq t \leq T\}$ for some $\varepsilon, T >0$ and use notation $V=V_{A,L}$.
   We consider the case $ n \geq 3$ first.
   Set $\xi =\lambda^{1/n}x, \tau = \lambda t$ then an easy computation leads to   $V(x,t)=\lambda^{(n-2)/n}V(\xi, \tau)$. Let $t_0 = \rho ^{-1}(\varepsilon /2)$. We split the proof into two cases:
   \begin{enumerate}[label=(\alph*)]
     \item   When $0\leq t \leq t_0$: Clearly from the formula,  we have $V(x,t)=0 \mbox{ in } \{(x,t): |x| \geq \varepsilon, 0 \leq t \leq t_0\}.$
     Besides,   for $\lambda$ large enough,
     \begin{align*}
     R_\lambda(t)= \frac{R(\lambda t)}{\lambda^{1/n}}\leq \frac{R(\lambda t_0)}{\lambda^{1/n}}<
     \rho(t_0)+\frac{\varepsilon}{2} = \varepsilon \mbox{  (due to Proposition \ref{free boundary
     bound}).}
     \end{align*}
     Thus, $\theta^\lambda =0=V$ in $\{(x,t): |x| \geq \varepsilon, 0 \leq t \leq t_0\}$ for $\lambda$ large enough.
     \item  When $t_0 \leq t \leq T$, we have:
     \begin{align}
     \label{second case}
     |\theta^\lambda(x,t) -V(x,t)| = t^{(2-n)/n} \tau^{(n-2)/n}|\theta(\xi, \tau) - V(\xi, \tau)|
     \end{align}
     Since $t_0 \leq t \leq T$, $t^{(2-n)/n} $ is bounded. From Theorem \ref{far field radial}, the right hand side of (\ref{second case}) converges to $0$ uniformly in the sets $\{\xi \in \mathbb{R}^n: |\xi| \geq \delta \tau^{1/n}\}=\{x \in \mathbb{R}^n: |x| \geq \delta t^{1/n}\} \supset \{(x,t): |x| \geq \varepsilon, t_0 \leq t \leq T\}$ for fixed $\varepsilon$ and $\delta>0$ small enough and thus we obtain the convergence for $n \geq 3$.
   \end{enumerate}
   For $n=2$, we argue similar as the case $n \geq 3$, but noting that
$\lim\limits_{\lambda \rightarrow \infty} \frac{\mathcal{R}(\tau)}{\mathcal{R(\lambda)}}= t^{1/2} $
together with Theorem \ref{far field radial}.
 \end{proof}
 \subsection{Some more results for viscosity solutions}

 Following \cite{P1,QV}, we also can state some results for viscosity solutions.
 \begin{lemma}
   For $L=1/m$ (resp. $L=1/M$), with $m,M$ as in (\ref{condition in g}), let $\theta(x,t)$ be the radial solution of Stefan problem (\ref{Stefan}) satisfying boundary conditions (\ref{boundary con1}) and initial condition (\ref{initial cond radial}) with $g(x)=1/L$ and $a$ such that $B(0,a) \subset K$ (resp. $K \subset B(0,a)$). Then the function $\theta(x,t)$ is a viscosity subsolution (resp. supersolution) of the Stefan problem (\ref{Stefan}) in $Q$.
 \end{lemma}
 \begin{proof}
The statement follows directly from properties of radially solutions and the fact that a classical solution is also a viscosity solution.
 \end{proof}
 Using viscosity comparison principle, we also can get the same estimates for free boundary as in
 Proposition \ref{free boundary bound} and boundedness for a \textbf{general viscosity solution}.
 \begin{lemma}
   \label{viscos free boundary bound}
   Let $v$ be a viscosity solution of (\ref{Stefan}). There exists $t_0>0$ and constant $C, C_1, C_2>0$ such that for $t\geq t_0$,
   \begin{align*}
   C_1 t^{1/n} &< \min_{\Gamma_t(v)}|x|\leq \max_{\Gamma_t(v)}|x|< C_2t^{1/n} &&\mbox{if } n\geq 3,\\
    C_1 \mathcal{R}(t)&<\min_{\Gamma_t(v)}|x|\leq \max_{\Gamma_t(v)}|x|< C_2 \mathcal{R}(t) &&\mbox{if } n= 2,
   \end{align*}
   and for $0 \leq t\leq t_0$,
   $\displaystyle\max_{\Gamma_t(v)}|x|< C_2.$ Moreover,
   $0 \leq v(x,t) \leq C|x|^{2-n} \mbox{ for all } n \geq 2.$
 \end{lemma}
 \begin{proof}
   Argue as in \cite{P1} with using Lemma \ref{free boundary bound} and Lemma \ref{Boundedness of weak solution} above.
 \end{proof}
We also have the near field limit and the asymptotic behavior result as in \cite{QV}.
 \begin{theorem}[Near-field limit]
   \label{Near field limit Theorem} The viscosity solution $v(x,t)$ of the Stefan problem
   (\ref{Stefan}) converges to the unique solution $P(x)$ of the exterior Dirichlet problem
   \begin{equation}
   \label{near filed limit}
   \left\{
     \begin{aligned}
   \Delta P&=0, && x \in \mathbb{R}^n \backslash  K,\\
   P&=1,&&x \in \Gamma,\\
   \lim\limits_{|x| \rightarrow \infty}P(x) &=0&& \mbox{if } n \geq 3, &&
   \mbox{ or } & P \mbox{ is bounded }& \mbox{if } n=2,
 \end{aligned}
   \right.
   \end{equation}
   as $t \rightarrow \infty$ uniformly on compact subsets of $\overline{K^c}$.
 \end{theorem}
 \begin{proof}
   See proof of Theorem 8.1 in \cite{QV}.
 \end{proof}
 \begin{lemma}[cf. {\cite[Lemma 4.5]{QV}}]
   \label{C*}
   There exists a constant $C_*=C_*(K,n)$ such that the solution $P$ of problem (\ref{near filed limit}) satisfies
   $\lim\limits_{|x| \rightarrow \infty}|x|^{n-2}P(x)=C_*.$
 \end{lemma}
 \section{Uniform convergence of variational solutions}
 \label{sec:var-conv}
 \subsection{Limit problem and the averaging properties of media}
We first recall the limit variational problem as introduced in \cite{P1} (see \cite[section 5]{P1} for
derivation and properties). Let $U_{A,L}(x,t) := \int_0^t V_{A, L}(x, s) \;ds$. For given
$A,L>0$, \cite[Theorem~5.1]{P1} yields that $U_{A,L}(x,t)$ is the unique solution of the limit obstacle problem
 \begin{equation}
 \label{limit problem}
 \left\{\begin{aligned}
 w &\in \mathcal{K}_t,\\
 a(w,\phi) &\geq \langle-L,\phi\rangle, && \text{for all } \phi \in V,\\
 a(w,\psi w)&= \langle L, \psi w\rangle && \text{for all } \psi \in W,
 \end{aligned}\right.
 \end{equation}
 where
 $\mathcal{K}_t= \Big\{ \varphi \in \bigcap_{\varepsilon>0}H^1(\mathbb{R}^n \backslash
 B_\varepsilon) \cap C(\mathbb{R}^n \backslash B_\varepsilon): \varphi \geq 0, \lim\limits_{|x|
 \rightarrow 0}\frac{\varphi(x)}{U_{A,L}(x,t)}=1\Big \},$
 \begin{align}
 \label{set V}
 V&=\left\{ \phi \in H^1(\mathbb{R}^n):\phi \geq 0, \phi = 0 \mbox{ on } B_\varepsilon \mbox{ for some } \varepsilon >0\right\},\\
 \label{set W}
 W&=V \cap C^1(\mathbb{R}^n),
 \end{align}
 and
$$ a_{\Omega} (u,v):=\int_{\Omega}{Du \cdot Dv dx}, \quad \left<u,v\right>_\Omega :=
\int_{\Omega}uvdx.  $$
We omit the set $\Omega$  in the notation if $\Omega = \mathbb{R}^n$.

 We also recall the following application of the subadditive ergodic theorem.
 \begin{lemma}[cf. {\cite[Section 4, Lemma 7]{K2}, see also \cite{P1}}]
   \label{media}
   For given $g$ satisfying (\ref{condition in g 2}), there exists a constant, denoted by
  $\left<1/g\right>$, such that if $\Omega \subset \mathbb{R}^n$ is a bounded measurable set and if
  $\{u^\varepsilon\}_{\varepsilon>0} \subset L^2(\Omega)$ is a family of functions such that $u^\varepsilon \rightarrow u$ strongly in $L^2(\Omega)$ as $\varepsilon \rightarrow 0$, then
   \begin{equation*}
   \lim\limits_{\varepsilon\rightarrow 0} \int\limits_{\Omega}\frac{1}{g(x/\varepsilon,\omega)}u^\varepsilon(x)dx
   =\int\limits_{\Omega}\left<\frac{1}{g}\right>u(x)dx   \mbox{ a.e. $\omega \in A$}.
\end{equation*}
 \end{lemma}
 \subsection{Uniform convergence of rescaled variational solutions}
 Now we are ready to prove the first main result, similar to Theorem 6.2 in \cite{P1}.
 \begin{theorem}
   \label{convergence of variational solutions}
   Let $u$ be the unique solution of variational problem (\ref{Variational problem}) and
   $u^\lambda$ be its rescaling. Let $U_{A,L}$ be the unique solution of limit problem (\ref{limit problem}) where $A=C_*$ as in Lemma \ref{C*}, and $L=\left<1/g\right>$ as in Lemma \ref{media}. Then the functions $u^\lambda$ converges locally uniformly to $U_{A,L}$ as $\lambda \rightarrow \infty$ on $\left(\mathbb{R}^n \backslash \{0\}\right) \times \left[0,\infty\right)$.
 \end{theorem}
 \begin{proof}
   We argue as in \cite{P1}. Fix $T > 0$. By Lemma~\ref{viscos free boundary bound}, we can bound
   $\Omega_t(u^\lambda)$ by $\Omega := B_{\delta}(0)$ for some $\delta >0$, for all $0 \leq t \leq
   T$ and $\lambda >0$. For some $\varepsilon >0$, define $\Omega_\varepsilon:= \Omega \backslash
   \overline{B(0,\varepsilon)}$, $Q_{\varepsilon}:= \Omega_\varepsilon \times [0,T]$ . We will prove the convergence in $Q_\varepsilon$.

   Let $v$ be the viscosity solution of the Stefan problem (\ref{Stefan}). We can find constants
   $0 < a < b$ such that $K \subset B_a(0)$ and $\overline{\Omega_0} \subset B_b(0)$. Set $L=1/M$ and $A
   = \max v_0$. Choose radially symmetric smooth $\theta_0 \geq 0$ such that $\theta_0 \geq v_0$ on
   $\Omega_0 \setminus B_a(0)$ and $\theta_0 = 0$ on $\Rn \setminus B_b(0)$. The radial solution
   $\theta$ of the
   Stefan problem on $\Rn \setminus \overline{B_a(0)}$ with such parameters will be above $v$ by the comparison
   principle.
   Thus, for $\lambda$ large enough, the rescaled solutions satisfy
   $$0 \leq  v^\lambda \leq  \theta^\lambda \mbox{ in } Q_{\varepsilon / 2}.$$
   On the other hand, by Lemma \ref{convergence radial lemma},
   $\theta^\lambda$ converges to $V_{A,L}$ as $\lambda \rightarrow
   \infty$ uniformly on $Q_{\varepsilon/2}$ and  $V_{A,L}$ is bounded in $Q_{\varepsilon/2}$ and therefore for $\lambda$ large enough so
   that $(B_a(0))^\lambda \subset B_{\varepsilon/2}(0)$,
   \begin{equation}
   \label{u^lamda_t bounded}
   \|u^\lambda_t\|_{L^\infty(Q_{\varepsilon/2})}= \|v^\lambda\|_{L^\infty(Q_{\varepsilon/2})}\leq C(\varepsilon).
   \end{equation}
   Since $u^\lambda$ satisfies (\ref{rescaled Variational problem}), we have
   $$\Delta u^\lambda (\varphi -u^\lambda) \leq
   \left(\lambda^{(2-n)/n}u^\lambda_t-f(\lambda^{1/n}x)\right)(\varphi - u^\lambda ) \mbox{ a.e for
   any } \varphi \in \mathcal{K}^\lambda(t).$$
   As $u^\lambda_t$ is bounded, $u^\lambda$ satisfies the elliptic obstacle problem
   $$\Delta u^\lambda (\varphi -u^\lambda) \leq \left(C\lambda^{(2-n)/n}-f(\lambda^{1/n}x)\right)(\varphi - u^\lambda )$$
    a.e for any $\varphi \in \mathcal{K}^\lambda(t)$ such that $\varphi - u^\lambda \geq 0$.

   Now we can use the standard regularity estimates for the obstacle problem (see
    \cite [Proposition 2.2, chapter 5]{R1} for instance),
   $$\displaystyle \|\Delta u^\lambda (\cdot,t)\|_{L^p(\Omega_{\varepsilon / 2})} \leq \left
   \|C\lambda^{(2-n)/n}-\frac{1}{g^\lambda} \right\|_{L^p(\Omega_{\varepsilon /2})} \leq C_0 \mbox{
   for all } 1 \leq p \leq \infty,$$
   for all $\lambda$ large so that also $\Omega_0^\lambda \subset B_{\varepsilon/2}(0)$.
   Using \eqref{u^lamda_t bounded} and $u^\lambda(x,t)= \int_{0}^{t}v^\lambda(x,s)ds$, we conclude
    $\|u^\lambda(\cdot,t)\|_{L^p(\Omega_{\varepsilon /2})}$ is bounded uniformly in $t \in [0,T]$
    and $\lambda$ large.

   Using elliptic interior estimate results for obstacle problem again (for example,
   \cite[Theorem 2.5]{R1}), we can find constants $0<\alpha < 1$ and $C_2$, independent of $t \in [0,T]$ and
   $\lambda \gg 1$, such that
   \begin{equation*}
     \begin{aligned}
       \|u^\lambda(\cdot,t)\|_{W^{2,p}(\Omega_\varepsilon)} & \leq C_2,\\
       \|u^\lambda(\cdot,t)\|_{C^{0,\alpha}(\Omega_\varepsilon)}& \leq C_2,
 \end{aligned}
   \quad\mbox{for all } 0 \leq t \leq T, \lambda \gg 1.
   \end{equation*}
   Moreover, using (\ref{u^lamda_t bounded}) again, we have $|u^\lambda(x,t)-u^\lambda(x,s)| \leq
   C_3|t-s|$. Thus $u^\lambda$ is H\"older continuous in $x$ with $0<\alpha<1$ and Lipschitz
   continuous in $t$. In particular, $u^\lambda$ satisfies
   $$\|u^\lambda\|_{C^{0,\alpha}(Q_\varepsilon)} \leq C_4(C_2,C_3) \mbox{ for all } \lambda \geq \lambda_0.$$

   The argument for case $n=2$ is similar.

   By the Arzel\`{a}-Ascoli theorem, we can find a function $\bar u \in C((\Rn \setminus \{0\})
   \times [0, \infty))$ and a subsequence $\{u^{\lambda_k}\} \subset
  \{u^\lambda\}$ such that
   $$u^{\lambda_k} \rightarrow \bar{u} \mbox{ locally uniformly on $(\Rn \setminus \{0\})
   \times [0, \infty)$} \mbox{ as } k
   \rightarrow \infty, $$
   Due to the compact embedding of $H^2$ in $H^1$, we have,
   $u^{\lambda_k}(\cdot,t) \rightarrow \bar{u}(\cdot,t)$ strongly in $H^1(\Omega_\varepsilon)$ for all $t
   \geq 0$, $\varepsilon > 0$.

   To finish the proof, we need to show that the function $\bar{u}$ is the solution of limit
   problem (\ref{limit problem}) and then by the uniqueness of the limit problem, we deduce that the convergence is not restricted to a subsequence.
   \begin{lemma}[cf. {\cite[Lemma 6.3]{P1}}]
     \label{U satisfies obstacle inequality}
     For each $t \geq 0$, $\bar{w}:=\bar{u}(\cdot,t)$ satisfies
     \begin{align}
     \label{1}
     a(\bar{w}, \phi) \geq \left<-L,\phi\right> & \mbox{ for all } \phi \in V,\\
     \label{2}
     a(\bar{w}, \psi \bar{w})=\left<-L, \psi \bar{w}\right>& \mbox{ for all } \psi \in W,
     \end{align}
     where $L= \left<1/g \right>$ as in Lemma \ref{media} and $V, W$ as in (\ref{set V}) and (\ref{set W}).
   \end{lemma}
   \begin{proof}Consider $n \geq 3$.
     Follow techniques in \cite{P1}, fix $t \in [0,T]$ and $\phi \in V$, denote $w^k:=
     u^{\lambda_k}(\cdot,t)$. Take $\phi \in V$ first. There exists $k_0>0$ such that for all $k
     \geq k_0$, $\Omega_0^{\lambda_k} \subset B_\varepsilon(0)$ and $\phi =0$ on
     $B_\varepsilon(0)$. Set $\varphi^k=\phi + w^k \in \mathcal{K}_t^{\lambda_k}$. Substitute the
     function $\varphi^k$ into the rescaled equation (\ref{rescaled Variational problem}),
     integrate both sides and integrate by parts, which yields
       $$a(w^{\lambda_k},\phi) \geq -\lambda^{(2-n)/n}_k \left<u^{\lambda_k}_t, \phi\right>+ \left<-\frac{1}{g^{\lambda_k}}, \phi\right>.$$
        Recalling Lemma \ref{media} and that $u^{\lambda_k}_t$ is bounded, $w \mapsto a(w,\phi)$ is
        bounded linear functional in $H^1$. Since $w^k \rightarrow \bar{w}$ strongly in $H^1$ as $k
        \rightarrow \infty$, we can send $\lambda_k \to \infty$ and obtain (\ref{1}).

     Now take $\psi \in W$ such that $0 \leq \psi \leq 1, \psi =0 $ on $B_\varepsilon(0)$ and take
     $k_0$ such that $\Omega^{\lambda_k}_0 \subset B_\varepsilon(0)$ for all $k \geq k_0$. Since
     $\psi \in W$ then $\psi \bar{w} \in V$. As above we have $a(\bar{w},\psi \bar{w}) \geq
     \left<-L, \psi \bar{w}\right>$. Moreover, consider $\varphi^k=(1-\psi)w^k \in \varphi^k \in
     \mathcal{K}^{\lambda_k}(t)$, $k \geq k_0$. Then,
     \begin{align*}
     a(w^k, \psi w^k)= -a(w^k, \varphi^k-w^k)&\leq\left<-\frac{1}{g^{\lambda_k}}, \psi w^k\right>+\lambda_k^{(n-2)/n}\left<w^k, \psi w^k\right>.
     \end{align*}
     Again using Lemma \ref{media}, boundedness in $L^\infty(\mathbb{R}^n)$ of $w^k$, the lower
     semi-continuity in $H^1$ of the map $w \mapsto a(w, \psi w)$, and the fact that $w^k \rightarrow \bar{w}$ strongly in $H^1$ as $k \rightarrow \infty$
     we can conclude the equality (\ref{2}).

     Again, $n = 2$ is similar.
   \end{proof}
   Finally, the next lemma establishes that the singularity of $\bar u$ as $|x| \rightarrow 0$ is correct.
   \begin{lemma}[cf. {\cite[Lemma 6.4]{P1}}]
     \label{singularity U}
     We have
     $$\lim\limits_{|x| \rightarrow 0} \frac{\bar{u}(x,t)}{U_{C_*, L}}(x,t)=1$$
     for every $t \geq 0$, where $C_*$ as in Lemma \ref{C*}.
   \end{lemma}

   \begin{proof}
     The proof follows the proof of \cite[Lemma 6.4]{P1} since the solutions of the Stefan problem have the same near field limit (Lemma \ref{near filed limit}) as the Hele-Shaw solutions.
   \end{proof}
   This finishes the proof of Theorem \ref{convergence of variational solutions} .
 \end{proof}

 \section{Uniform convergence of rescaled viscosity solutions and free boundaries}
 \label{sec:visc-conv}

 In this section, we will deal with the convergence of $v^\lambda$ and their free boundaries. Let $v$ be a viscosity solution of the Stefan problem (\ref{Stefan}) and $v^\lambda$ be its rescaling.  Let $V=V_{C_*,L}$ be the solution of Hele-Shaw problem with a point source as in (\ref{Hele-Shaw with point sorce}), where $C_*$ is the constant of Lemma \ref{C*} and $L=\left<1/g\right>$ as in Lemma \ref{media}.

 We define the half-relaxed limits  in $\{|x| \neq 0, t\geq 0\}$:
 \begin{align*}
 & v^*(x,t)= \limsup_{(y,s),\lambda \rightarrow (x,t), \infty} v^\lambda(y,s), & v_*(x,t)= \liminf_{(y,s),\lambda \rightarrow (x,t), \infty} v^\lambda(y,s),
 \end{align*}
 \begin{remark}
 $V$ is continuous in $\{|x| \neq 0, t\geq 0\}$, therefore $V_*=V=V^*$.
 \end{remark}
 To complete Theorem~\ref{th:main-convergence}, we prove a result similar to \cite[Theorem 7.1.]{P1}
 \begin{theorem}
   \label{convergence of rescaled viscosity solution}
   The rescaled viscosity solution $v^\lambda$ of the Stefan problem (\ref{Stefan}) converges
   locally uniformly to $V = V_{C_*, \ang{1/g}}$ in $(\mathbb{R}^n \backslash \{0\}) \times [0,\infty)$ as $\lambda \rightarrow \infty$ and
   $$v_*=v^*=V.$$
   Moreover, the rescaled free boundary $\{\Gamma(v^\lambda)\}_{\lambda}$ converges to $\Gamma(V)$ locally uniformly with respect to the Hausdorff distance.
 \end{theorem}

 To prepare for the proof of Theorem \ref{convergence of rescaled viscosity solution}, we need to
 collect some results which are similar to the ones in \cite{K3} and \cite{P1} with some adaptations to
 our case.  All the results for $n\geq 3$ we have in this section can be obtained for $n=2$ by
 using limit $\frac{1}{\log \mathcal{R}(\lambda)} \rightarrow 0$ as $\lambda \rightarrow \infty$.
 Thus, from here on we only consider case $n \geq 3$, the results for $n=2$ are omitted.
 \subsection{Some necessary technical results}
 \begin{lemma}[cf. {\cite[Lemma 3.9]{K3}}]
   \label{domain of u coincide with domain of v}
   The viscosity solution $v$ of the Stefan problem (\ref{Stefan}) is strictly positive in $\Omega(u)$, satisfies $\Omega(v)= \Omega(u) \mbox{ and } \Gamma(v) = \Gamma(u)$.

 \end{lemma}

 \begin{lemma}
   \label{subharmonic v^* Lemma}
   Let $v^\lambda$ be a viscosity solution of the rescaled problem \eqref{rescaled equation}.
   Then $v^*(\cdot,t)$ is subharmonic in $\Rn \setminus \{0\}$ and $v_{*}(\cdot,t)$ is superharmonic in $\Omega_t(v_*) \backslash \{0\}$ in viscosity sense.
 \end{lemma}
 \begin{proof}
   This follows from a standard viscosity solution argument using test functions, see for instance \cite{K1}.
 \end{proof}

 The behavior of functions $v^*,v_*$ at the origin and their boundaries can be established by
 following the arguments in \cite{P1} and \cite{K3}.
 \begin{lemma}[$v^*$ and $v_*$ behave as $V$ at the origin]
   \label{boundary condition for limit problem}
   The functions $ v^*, v_*$ have a singularity at $0$ with:
   \begin{align}
   \label{singularity of V}
   &\lim\limits_{|x| \rightarrow 0+}\frac{v_*(x,t)}{V(x,t)}=1, & \lim\limits_{|x| \rightarrow 0+}\frac{v^*(x,t)}{V(x,t)}=1, \mbox{ for each } t>0.
   \end{align}
 \end{lemma}
 \begin{proof}
   See \cite[Lemma 7.4]{P1}.
 \end{proof}
 \begin{lemma}[cf. {\cite[Lemma 5.4]{K3}}]
   \label{limit of sequence in 0-level set of ulambda k}
   Suppose that $(x_k,t_k) \in \{u^{\lambda_k}=0\}$ and $(x_k,t_k,\lambda_k) \rightarrow (x_0,t_0,\infty)$. Then:\begin{enumerate}[label=\alph*)]
     \item $U(x_0,t_0)=0$,
     \item If $x_k \in \Gamma_{t_k}(u^{\lambda_k})$ then $x_0  \in \Gamma
     _{t_0}(U)$.
   \end{enumerate}
 \end{lemma}
 \begin{proof}
   See proof of \cite[Lemma 5.4]{K3}.
 \end{proof}

 The rest of the convergence proof in \cite{P1} relies on the monotonicity of the solutions of the
 Hele-Shaw problem in time. Since the Stefan problem lacks this monotonicity, we will show that
 sufficiently regular initial data satisfy a weak monotonicity below. The convergence result for general initial data will then follow by uniqueness of the limit and the comparison principle.

 \begin{lemma}
 	\label{weak monotonicity}
  Suppose that $v_0$ satisfies \eqref{initial data}. Then there exist $C \geq 1$ independent of $x$ and $t$ such that
  \begin{equation}
  \label{monotonicity condition}
  v_0(x) \leq C v(x, t) \mbox{ in } \mathbb{R}^n \backslash K \times [0,\infty).
  \end{equation}
 \end{lemma}
\begin{proof}
  Let $\gamma_1:=\min_{\partial \Omega_0}|Dv_0|, \gamma_2 :=\max_{\partial \Omega_0}|Dv_0|$. Note
  that $0<\gamma_1\leq \gamma_2 <\infty$. For given $\varepsilon > 0$, let $w$ be the solution of
  boundary value problem
\begin{equation*}
\left \{
  \begin{aligned}
\Delta w &= 0 &&\mbox{in } \Omega_0 \backslash K,\\
w &= \varepsilon &&\mbox{on } K,\\
w &=0 && \mbox{on } \Omega_0^c.
\end{aligned}
\right.
\end{equation*}
For $x$ close to $\partial \Omega_0$ we have $v_0(x) \geq \frac{\gamma_1}{2} \text{dist}(x,\partial
\Omega_0)$. Since $\gamma_1>0$, $v_0 >0$ in $\Omega_0$ and $\partial \Omega_0$ has a uniform ball
condition, we can choose $\varepsilon >0$ small enough such that $w \leq v_0$ in $\mathbb{R}^n
\backslash K$. By Hopf's Lemma, $\gamma_w:=\min_{\partial\Omega_0}|Dw|>0$. It is clear that $w$ is
a classical subsolution of the Stefan problem (\ref{Stefan}) and the comparison principle yields
\begin{equation}
\label{w<=v}
w \leq v \mbox{ in } (\mathbb{R}^n \backslash K) \times [0,\infty).
\end{equation}

 Now assume that (\ref{monotonicity condition}) does not hold, that is,  for every $k \in \mathbb N$, there exists $(x_k,t_k) \in \mathbb{R}^n \backslash K \times [0,\infty)$ such that
\begin{equation}
\label{contradiction}
\frac{1}{k}v_0(x_k)> v(x_k,t_k).
\end{equation}
Clearly $x_k \in \Omega_0$. $\{t_k\}$ is bounded by Theorem~\ref{Near field limit Theorem} since
$v_0$ is bounded.
Therefore, there exists a subsequence $(x_{k_l},t_{k_l})$ and a point $(x_0,t_0)$ such that
$(x_{k_l},t_{k_l}) \rightarrow (x_0,t_0)$. Since $v_0$ is bounded, we get $v(x_0,t_0)\leq 0$ and
thus $x_0 \in \partial \Omega_0$ by \eqref{w<=v}. Consequently, for $k_l$ large enough,
\begin{equation*}
w(x_{k_l}) \geq \frac{1}{2} \gamma_w \text{dist}(x_{k_l}, \partial
\Omega_0)=\left(\frac{\gamma_w}{4\gamma_2}\right)2\gamma_2 \text{dist}(x_{k_l}, \partial \Omega_0)
\geq \frac{\gamma_w}{4 \gamma_2} v_0(x_{k_l}).
\end{equation*}
Combine this with (\ref{w<=v}) and (\ref{contradiction}) to obtain
\begin{equation*}
\frac{1}{k_l}v_0(x_{k_l})> \frac{\gamma_w}{4\gamma_2} v_0(x_{k_l})
\end{equation*}
for every $k_l$ large enough, which yields a contradiction since $v_0(x_{k_l})>0$.
 		\end{proof}

 Some of the following lemmas will hold under the condition (\ref{monotonicity condition}).
 \begin{lemma}
   \label{monotonicity 2}
   Let $u$ be the solution of the variational problem \eqref{Variational problem}, and $v$ be the
   associated viscosity solution of the Stefan
  problem, and suppose that \eqref{monotonicity condition} holds. Then
   \begin{equation}
   u(x,t)\leq C t v(x,t).
   \end{equation}
 \end{lemma}
 \begin{proof}
   The statement follows from checking that $\tilde{u}:= C tv$ is a supersolution of the heat
   equation in $\Omega(u)$ and the classical comparison principle. Indeed,
   $\tilde u_t - \Delta \tilde u = Cv + Ct(v_t - \Delta v) \geq v_0 \geq f = u_t - \Delta u$ in $\Omega(u)$ by
   \eqref{monotonicity condition}.
   \end{proof}
   \begin{lemma}[cf. {\cite[Lemma 5.5]{K3}}]
   \label{Inclusion of positive domain V, v*}
   The function $v_*$ satisfies
   $\Omega(V) \subset \Omega(v_*).$
   In particular $v_* \geq V$.
 \end{lemma}
 \begin{proof}
   Assume that the inclusion does not hold, there exists $(x_0,t_0) \in \Omega(V)$ and $v_*(x_0,t_0)=0$.
   By (\ref{monotonicity condition}) and Lemma \ref{monotonicity 2}, there exists $C>1$ such that $
   u(x,t)\leq C t v(x,t) .$ This inequality is preserved under the rescaling, $u^\lambda(x,t)\leq C
   t v^\lambda(x,t)$ in $(\mathbb{R}^n\backslash K^\lambda)\times [0,\infty)$. Taking $\liminf^*$
   of both sides gives the contradiction $0<U(x_0,t_0) \leq C t_0v_*(x_0,t_0)=0$.

   The inequality $v_* \geq V$ follows from the elliptic comparison principle as $v_*$ is superharmonic in
   $\Omega(v_*) \setminus \{0\}$ by Lemma~\ref{subharmonic v^* Lemma} and behaves as $V$ at the origin by Lemma \ref{boundary condition for limit
  problem}.
 \end{proof}

 \begin{lemma}
   \label{positive sup rescaling}
 There exists constant $C > 0$ independent of $\lambda$ such that for every $x_0 \in \overline{\Omega_{t_0}(u^\lambda)}$ and $B_r(x_0) \cap \Omega_0^{\lambda} = \emptyset$ for some $r$, for every $\lambda$ large enough we have
   \begin{equation*}
     \sup_{x \in \overline{B_r(x_0)}}u^\lambda(x,t_0)>C r^2.
   \end{equation*}
 \end{lemma}
 \begin{proof}
   Follow the arguments in \cite[Lemma 3.1]{K2} with noting that since $u^\lambda_t$ is bounded then
  for $\lambda$ large enough, $u^\lambda$ is a strictly subharmonic function in
  $\Omega_{t_0}(u^\lambda) \setminus \overline{\Omega_0^\lambda}$.
   \end{proof}
 \begin{cor}
   \label{monotoniciy for v}
   There exists a constant $C_1=C_1(n,M,\lambda_0)$ such that if $(x_0,t_0) \in \Omega(v^\lambda)$ and $B_r(x_0) \cap \Omega_0^\lambda = \emptyset$ and $\lambda\geq \lambda_0$, we have
   \begin{equation*}
   \sup_{B_r(x_0)}v^\lambda(x,t_0) \geq \frac{C_1r^2}{t_0}.
   \end{equation*}
 \end{cor}
 \begin{proof}
   The inequality follows directly from Lemma \ref{monotonicity 2} and Lemma \ref{positive sup rescaling}.
 \end{proof}
 \begin{lemma}[cf. {\cite[Lemma 5.6 ii]{K3}}]
   \label{Inclusion of boundaries}
   We have the following inclusion:
   $$\Gamma(v^*) \subset \Gamma(V).$$
 \end{lemma}
 \begin{proof}
   Argue as in  \cite[Lemma 5.6 ii]{K3} together with using Lemma \ref{limit of sequence in 0-level set of ulambda k} and Lemma \ref{positive sup rescaling} above.
   \end{proof}

 Now we are ready to prove Theorem \ref{convergence of rescaled viscosity solution}.
 \subsection{Proof of Theorem \ref{convergence of rescaled viscosity solution}}
 \begin{proof}
   \textit{\textbf{Step 1}}. We prove the convergence of viscosity solutions and the free boundaries under the conditions (\ref{initial data}) and (\ref{monotonicity condition}) first.

   Lemma~\ref{viscos free boundary bound} yields that $\Omega_t(v^*)$ is bounded at all time $t>0$. Since $\Omega(V)$ is simply connected set, Lemma \ref{Inclusion of boundaries} implies that
   \begin{equation*}
   \overline{\Omega(v^*)} \subset \overline{\Omega(V)} \subset \Omega(V_{C^*+\varepsilon,L}) \mbox{ for all } \varepsilon>0.
   \end{equation*}
   We see from Lemma \ref{subharmonic v^* Lemma}, $v^*(\cdot, t)$ is a subharmonic function in $\mathbb{R}^n \setminus \{0\}$ for every $t>0$ and $\lim_{|x|\rightarrow 0}\frac{v^*(x,t)}{V(x,t)}=1$ for all $t \geq 0$ by Lemma \ref{boundary condition for limit problem}, comparison principle yields $v^*(x,t) \leq V_{C_*+\varepsilon,L}(x,t)$ for every $\varepsilon>0$.

   By Lemma \ref{Inclusion of positive domain V, v*}, $V(x,t) \leq v_*$ and letting $\varepsilon
   \rightarrow 0^+$ we obtain by continuity
   $$V(x,t) \leq v_*(x,t) \leq v^*(x,t) \leq V(x,t).$$
   Therefore, $v_*=v^*= V$ and in particular, $\Gamma(v_*)= \Gamma(v^*)= \Gamma(V)$.

     Now we need to show the uniform convergence of the free boundaries with respect to the Hausdorff distance. Fix $0 < t_1< t_2$ and denote:
     \begin{align*}
     & \Gamma^\lambda:= \Gamma(v^\lambda) \cap \{t_1 \leq t \leq t_2\}, & \Gamma^\infty:= \Gamma(V) \cap \{t_1 \leq t \leq t_2\},
     \end{align*}
a $\delta$-neighborhood of a set $A$ in $\mathbb{R}^n \times \mathbb{R}$   is  $$U_\delta(A):= \{(x,t): \mbox{dist}((x,t), A) < \delta\}.$$

  We need to prove that for all $\delta >0$, there exists $\lambda_0>0$ such that:
  \begin{equation}
  \label{1st Theorem 4.2}
  \begin{matrix}
  \Gamma^\lambda \subset U_\delta (\Gamma^\infty) & \mbox{ and } &
  \Gamma^\infty \subset U_\delta(\Gamma^\lambda), \hspace{1cm} \forall \lambda \geq \lambda_0.
  \end{matrix}
  \end{equation}
  We prove the first inclusion in (\ref{1st Theorem 4.2}) by contradiction. Suppose
  therefore that we can find a subsequence $\{\lambda_k\}$ and a sequence of points $(x_k,t_k) \in \Gamma^{\lambda_k}$ such that
  $\mbox{dist}((x_k,t_k), \Gamma^\infty) \geq \delta.$
  Since $\Gamma^\lambda$ is uniformly bounded in $\lambda$ by Lemma \ref{viscos free boundary
  bound}, there exists a subsequence $\{(x_{k_j}, t_{k_j})\}$ which converge to a point
  $(x_0,t_0)$. By Lemma \ref{limit of sequence in 0-level set of ulambda k}, $(x_0,t_0) \in
  \Gamma(U)= \Gamma(V)$. Moreover, since $t_1 \leq t_{k_j} \leq t_2$ then $t_1 \leq t_0 \leq t_2$
  and therefore, $(x_0,t_0) \in \Gamma^\infty$, a contradiction.

  The proof of the second inclusion in (\ref{1st Theorem 4.2}) is more technical. We prove a
  pointwise result first. Suppose that there exists $\delta >0, (x_0,t_0) \in \Gamma^\infty$ and
  $\{\lambda_k\}, \lambda_k \rightarrow \infty$, such that $\displaystyle \mbox{dist}((x_0,t_0),
  \Gamma^{\lambda_k}) \geq \frac{\delta}{2}$ for all $k$. Then there exists $r >0$ such that $D_r(x_0,t_0):= B(x_0, r) \times [t_0-r,t_0+r]$ satisfies either:
  \begin{equation}
  \label{4nd Theorem 4.2}
  D_r(x_0,t_0) \subset \{v^{\lambda_k}=0\}  \mbox{ for all } k,
  \end{equation}
  or after passing to a subsequence,
  \begin{equation}
  \label{3rd Theorem 4.2}
  D_r(x_0,t_0) \subset \{v^{\lambda_{k}}>0\}  \mbox{ for all } k.
  \end{equation}
If (\ref{4nd Theorem 4.2})  holds, clearly $V=v_* =0$  in $D_r(x_0,t_0)$ which is in a contradiction with the assumption that $(x_0,t_0) \in \Gamma^\infty$.

Thus we assume that  (\ref{3rd Theorem 4.2}) holds.  In $D_r(x_0,t_0)$, $v^{\lambda_k}$ solves the
heat equation $\lambda^{(2-n)/n}v^{\lambda_k}_t- \Delta v^{\lambda_k} =0$. Set
\begin{equation*}
w^k(x,t):=v^{\lambda_k}(x,\lambda_k^{(2-n)/n} t)
\end{equation*} then $w^k>0$ in $D_r^w(x_0,t_0):=B(x_0,r) \times
[\lambda_k^{(n-2)/n}(t_0-r),\lambda_k^{(n-2)/n}(t_0+r)]$ and $w^k$ satisfies $w^k_t -\Delta w^k=0$
in $D_r^w(x_0,t_0)$. Since $\lambda_k^{(n-2)/n} \frac r2 \to \infty$ as $k \to \infty$, by Harnack's inequality for the heat equation, for fixed $\tau > 0$ there exists a constant $C_1 >
0$ such that for each $t \in [t_0-\frac r2, t_0+\frac r2]$ and $\lambda_k$ such that $\tau <
\lambda_k^{(n-2)/n} \tfrac r4$ we have
\begin{equation*}
\sup_{B(x_0,r/2)} w^k(\cdot, \lambda_k^{(n-2)/n}t- \tau) \leq C_1 \inf_{B(x_0,r/2)}w^k(\cdot, \lambda_k^{(n-2)/n}t).
\end{equation*}
This inequality together with Corollary~\ref{monotoniciy for v} yields:
\begin{equation*}
\frac{C_2r^2}{t- \lambda_k^{(2-n)/n}\tau} \leq   \sup_{B(x_0,r/2)} v^{\lambda_k}(\cdot, t- \lambda_k^{(2-n)/n}\tau) \leq C_1 \inf_{B(x_0,r/2)}v^{\lambda_k}(\cdot, t)
\end{equation*}
for all $t \in [t_0 - \frac r2, t_0 + \frac r2]$, $\lambda_k \geq \lambda_0$ large enough, where
$C_2$ only depends on $n, M, \lambda_0$. Taking the limit when $\lambda_k \rightarrow \infty$, the
uniform convergence of $\{v^{\lambda_k}\}$ to $V$ gives $V>0$ in $B(x_0,\frac r2) \times [t_0-\frac
r2,t_0+\frac r2]$, which is a contradiction with $(x_0,t_0) \in \Gamma^\infty \subset \Gamma(V)$.

We have proved that every point of $\Gamma^\infty$ belongs to all
$U_{\delta/2}(\Gamma^\lambda)$ for sufficiently large $\lambda$. Therefore the second
inclusion in \eqref{1st Theorem 4.2} follows from the compactness of $\Gamma^\infty$.

This concludes the proof of Theorem~\ref{convergence of rescaled viscosity solution} when
(\ref{monotonicity condition}) holds.

\textit{\textbf{Step 2}}. For general initial data, we will find upper and lower bounds for the initial
data for which (\ref{monotonicity condition}) holds, and use the comparison principle.
For instance, assume that $v_0 \in C(\Rn)$, $v_0 \geq 0$, such that $\supp v_0$ is bounded, $v_0 = 1$ on $K$.

Choose smooth bounded domains $\Omega_0^1, \Omega_0^2$ such that $K \subset \Omega_0^1
\subset \overline{\Omega_0^1} \subset \supp v_0 \subset \Omega_0^2$. Let $v_0^1, v_0^2$ be two
functions satisfying \eqref{initial data} with positive domains $\Omega_0^1,\Omega_0^2$,
respectively, and $v_0^1 \leq v_0 \leq v_0^2$. If necessary, that is, when $v_0$ is not
sufficiently regular at $\partial K$, we may perturb the boundary data for
$v_0^1$, $v_0^2$ on $K$ as $1 - \varepsilon$ and $1 + \varepsilon$, respectively, for some
$\varepsilon \in (0, 1)$.

Let $v_1, v_2$ be respectively the viscosity solution of the Stefan problem (\ref{Stefan}) with
initial data $v_0^1,v_0^2$. By the comparison principle, we have $v_1 \leq
v \leq v_2$ and after rescaling $v_1^\lambda \leq v^\lambda \leq v_2^\lambda$. By Step 1, we see
that $v_1^\lambda \to V_{C_{*,1-\varepsilon}, L}$ and $v_2^\lambda \to V_{C_{*,1+\varepsilon}, L}$.
Since $C_{*,1\pm\varepsilon} \to C_*$ as $\varepsilon \to 0$ by \cite[Lemma~4.5]{QV}, we deduce the local uniform convergence of
$v^\lambda \to V = V_{C_*, L}$.

The convergence of free boundaries follow from the ordering $\Omega(v_1) \subset \Omega(v) \subset
\Omega(v_2)$ and the convergence of free boundaries of $V_{C_{*,1\pm\varepsilon}, L}$ to
the free boundary of $V_{C_*,L}$ locally uniformly with respect to the Hausdorff distance.
\end{proof}

\subsection*{Acknowledgments}
The first author was partially supported by JSPS KAKENHI Grant No. 26800068 (Wakate B). This work
is a part of doctoral research of the second author. The second author would like to thank her
Ph.D.
supervisor Professor Seiro Omata for his valuable support and advice.


\begin{bibdiv}
  \begin{biblist}
\bib{Baiocchi}{article}{
   author={Baiocchi, Claudio},
   title={Sur un probl\`eme \`a fronti\`ere libre traduisant le filtrage de
   liquides \`a travers des milieux poreux},
   language={French},
   journal={C. R. Acad. Sci. Paris S\'er. A-B},
   volume={273},
   date={1971},
   pages={A1215--A1217},
   review={\MR{0297207}},
}

	\bib{C}{article}{
		author={L.A. Caffarelli},
		title={The regularity of free boundaries in higher dimensions},
		date={1977},
		journal={Acta Math.},
		pages={155\ndash 184},
		number={139},
		review={MR 56:12601}
	}

	\bib{CF}{article}{
		author={L.A. Caffarelli},
		author={A. Friedman},
		title={Continuity of the temperature in the Stefan problem},
		date={1979},
		journal={Indiana Univ. Math. J.},
		pages={53\ndash 70},
		number={28},
		review={MR 80i:35104},
	}

\bib{CS}{article}{
   author={Caffarelli, Luis A.},
   author={Souganidis, Panagiotis E.},
   title={Rates of convergence for the homogenization of fully nonlinear
   uniformly elliptic pde in random media},
   journal={Invent. Math.},
   volume={180},
   date={2010},
   number={2},
   pages={301--360},
   issn={0020-9910},
   review={\MR{2609244}},
   doi={10.1007/s00222-009-0230-6},
}


\bib{CSW}{article}{
   author={Caffarelli, Luis A.},
   author={Souganidis, Panagiotis E.},
   author={Wang, L.},
   title={Homogenization of fully nonlinear, uniformly elliptic and
   parabolic partial differential equations in stationary ergodic media},
   journal={Comm. Pure Appl. Math.},
   volume={58},
   date={2005},
   number={3},
   pages={319--361},
   issn={0010-3640},
   review={\MR{2116617}},
   doi={10.1002/cpa.20069},
}

\bib{Duvaut}{article}{
   author={Duvaut, Georges},
   title={R\'esolution d'un probl\`eme de Stefan (fusion d'un bloc de glace \`a
   z\'ero degr\'e)},
   language={French},
   journal={C. R. Acad. Sci. Paris S\'er. A-B},
   volume={276},
   date={1973},
   pages={A1461--A1463},
   review={\MR{0328346}},
}
  \bib{EJ}{article}{
     author={Elliott, C. M.},
     author={Janovsk\'y, V.},
     title={A variational inequality approach to Hele-Shaw flow with a moving
     boundary},
     journal={Proc. Roy. Soc. Edinburgh Sect. A},
     volume={88},
     date={1981},
     number={1-2},
     pages={93--107},
     issn={0308-2105},
     review={\MR{611303}},
     doi={10.1017/S0308210500017315},
  }

\bib{Evans}{article}{
   author={Evans, Lawrence C.},
   title={The perturbed test function method for viscosity solutions of
   nonlinear PDE},
   journal={Proc. Roy. Soc. Edinburgh Sect. A},
   volume={111},
   date={1989},
   number={3-4},
   pages={359--375},
   issn={0308-2105},
   review={\MR{1007533}},
   doi={10.1017/S0308210500018631},
}

    \bib{FK}{article}{
     	author={Friedman, Avner},
     	author={Kinderlehrer, David},
     	title={A one phase Stefan problem},
     	journal={Indiana Univ. Math. J.},
     	volume={24},
     	date={1974/75},
     	number={11},
     	pages={1005--1035},
     	issn={0022-2518},
     	review={\MR{0385326}},
     	doi={10.1512/iumj.1975.24.24086},
     }

	\bib{FN}{article}{
			author={Kinderlehrer, David},
			author={Nirenberg, Louis},
			title={The smoothness of the free boundary in the one phase Stefan
				problem},
			journal={Comm. Pure Appl. Math.},
			volume={31},
			date={1978},
			number={3},
			pages={257--282},
			issn={0010-3640},
			review={\MR{480348}},
			doi={10.1002/cpa.3160310302},
		}

    \bib{K1}{article}{
      	author={Kim, Inwon C.},
      	title={Uniqueness and existence results on the Hele-Shaw and the Stefan
      		problems},
      	journal={Arch. Ration. Mech. Anal.},
      	volume={168},
      	date={2003},
      	number={4},
      	pages={299--328},
      	issn={0003-9527},
      	review={\MR{1994745}},
      	doi={10.1007/s00205-003-0251-z},
      }

\bib{KimHomog}{article}{
   author={Kim, Inwon C.},
   title={Homogenization of the free boundary velocity},
   journal={Arch. Ration. Mech. Anal.},
   volume={185},
   date={2007},
   number={1},
   pages={69--103},
   issn={0003-9527},
   review={\MR{2308859 (2008f:35019)}},
   doi={10.1007/s00205-006-0035-3},
}

\bib{KimContact}{article}{
   author={Kim, Inwon C.},
   title={Homogenization of a model problem on contact angle dynamics},
   journal={Comm. Partial Differential Equations},
   volume={33},
   date={2008},
   number={7--9},
   pages={1235--1271},
   issn={0360-5302},
   review={\MR{2450158 (2010e:35031)}},
   doi={10.1080/03605300701518273},
}

\bib{KimContactRates}{article}{
   author={Kim, Inwon C.},
   title={Error estimates on homogenization of free boundary velocities in
   periodic media},
   journal={Ann. Inst. H. Poincar\'e Anal. Non Lin\'eaire},
   volume={26},
   date={2009},
   number={3},
   pages={999--1019},
   issn={0294-1449},
   review={\MR{2526413 (2010f:35020)}},
   doi={10.1016/j.anihpc.2008.10.004},
}

    \bib{K2}{article}{
      	author={Kim, Inwon C.},
      	author={Mellet, Antoine},
      	title={Homogenization of a Hele-Shaw problem in periodic and random
      		media},
      	journal={Arch. Ration. Mech. Anal.},
      	volume={194},
      	date={2009},
      	number={2},
      	pages={507--530},
      	issn={0003-9527},
      	review={\MR{2563637}},
      	doi={10.1007/s00205-008-0161-1},
      }

    \bib{K3}{article}{
     	author={Kim, Inwon C.},
     	author={Mellet, Antoine},
     	title={Homogenization of one-phase Stefan-type problems in periodic and
     		random media},
     	journal={Trans. Amer. Math. Soc.},
     	volume={362},
     	date={2010},
     	number={8},
     	pages={4161--4190},
     	issn={0002-9947},
     	review={\MR{2608400}},
     	doi={10.1090/S0002-9947-10-04945-7},
     }

    \bib{P1}{article}{
      	author={Po\v z\'ar, Norbert},
      	title={Long-time behavior of a Hele-Shaw type problem in random media},
      	journal={Interfaces Free Bound.},
      	volume={13},
      	date={2011},
      	number={3},
      	pages={373--395},
      	issn={1463-9963},
      	review={\MR{2846016}},
      	doi={10.4171/IFB/263},
      }

\bib{Pozar15}{article}{
   author={Po\v z\'ar, Norbert},
   title={Homogenization of the Hele-Shaw problem in periodic spatiotemporal
   media},
   journal={Arch. Ration. Mech. Anal.},
   volume={217},
   date={2015},
   number={1},
   pages={155--230},
   issn={0003-9527},
   review={\MR{3338444}},
   doi={10.1007/s00205-014-0831-0},
}

    \bib{QV}{article}{
    	author={Quir\'os, Fernando},
    	author={V\'azquez, Juan Luis},
    	title={Asymptotic convergence of the Stefan problem to Hele-Shaw},
    	journal={Trans. Amer. Math. Soc.},
    	volume={353},
    	date={2001},
    	number={2},
    	pages={609--634},
    	issn={0002-9947},
    	review={\MR{1804510}},
    	doi={10.1090/S0002-9947-00-02739-2},
    }

    \bib{R3}{article}{
       author={Rodrigues, Jos\'e-Francisco},
       title={Free boundary convergence in the homogenization of the one-phase
       Stefan problem},
       journal={Trans. Amer. Math. Soc.},
       volume={274},
       date={1982},
       number={1},
       pages={297--305},
       issn={0002-9947},
       review={\MR{670933}},
       doi={10.2307/1999510},
    }

    \bib{R1}{book}{
   author={Rodrigues, Jos\'e-Francisco},
   title={Obstacle problems in mathematical physics},
   series={North-Holland Mathematics Studies},
   volume={134},
   note={Notas de Matem\'atica [Mathematical Notes], 114},
   publisher={North-Holland Publishing Co., Amsterdam},
   date={1987},
   pages={xvi+352},
   isbn={0-444-70187-7},
   review={\MR{880369}},
}

  \bib{RReview}{article}{
     author={Rodrigues, Jos\'e-Francisco},
     title={The Stefan problem revisited},
     conference={
        title={Mathematical models for phase change problems},
        address={\'Obidos},
        date={1988},
     },
     book={
        series={Internat. Ser. Numer. Math.},
        volume={88},
        publisher={Birkh\"auser, Basel},
     },
     date={1989},
     pages={129--190},
     review={\MR{1038069}},
  }
    \bib{R2}{article}{
    	author={Rodrigues, Jos\'e-Francisco},
    	title={Variational methods in the Stefan problem},
    	conference={
    		title={Phase transitions and hysteresis},
    		address={Montecatini Terme},
    		date={1993},
    	},
    	book={
    		series={Lecture Notes in Math.},
    		volume={1584},
    		publisher={Springer, Berlin},
    	},
    	date={1994},
    	pages={147--212},
    	review={\MR{1321833}},
    	doi={10.1007/BFb0073397},
    }
\bib{Souganidis}{article}{
   author={Souganidis, Panagiotis E.},
   title={Stochastic homogenization of Hamilton-Jacobi equations and some
   applications},
   journal={Asymptot. Anal.},
   volume={20},
   date={1999},
   number={1},
   pages={1--11},
   issn={0921-7134},
   review={\MR{1697831}},
}
  \end{biblist}
\end{bibdiv}
\end{document}